\documentclass[10pt, fleqn]{amsart}
\usepackage{fixmath}

\usepackage[T1]{fontenc}
\usepackage{ucs}
\usepackage[utf8x]{inputenc}
\usepackage[english]{babel}

\usepackage{microtype}
\usepackage{babel}

\usepackage{mathtools} 
\usepackage{geometry} 

\usepackage{amsfonts, amsmath, amsthm, amssymb}
\usepackage{hyperref} 
\usepackage{tikz}
\usetikzlibrary{shapes}
\tikzstyle{node}=[draw, circle, minimum size=5pt, fill=black, inner sep=0pt]

\newtheorem{theorem}{Theorem}

\newtheorem{lemma}{Lemma}

\newtheorem{corollary}{Corollary}

\theoremstyle{definition}

\theoremstyle{remark}
\newtheorem{remark}{Remark}

\DeclareRobustCommand{\lah}{\genfrac{\lfloor}{\rfloor}{0pt}{}}

\providecommand{\Prob}[1]{\mathbb{P}\{#1\}}
\providecommand{\card}[1]{\texttt{\#}#1}
\providecommand{\abs}[1]{\lvert#1\rvert}

\newcommand{\pset}[1]{\mathcal{P}_{#1}}

\newcommand{\Ex}{\mathbb{E}}

\newcommand{\psetn}{\mathcal{P}_{[n]}}

\newcommand{\res}[2]{#1|_#2}

\begin{document}
\title[Spectral decomposition for Bolthausen-Sznitman and Kingman coalescent]{A Spectral decomposition for the Bolthausen-Sznitman coalescent and the Kingman coalescent}
\author{Jonas Kukla}
\address{Mathematisches Institut\\
Eberhard Karls Universit\"at T\"ubingen\\
Auf der Morgenstelle 10\\
72076 T\"ubingen, Germany}
\email{jonas.kukla@uni-tuebingen.de}
\author{Helmut H.~Pitters}
\address{Department of Statistics\\
University of Oxford\\
1 South Parks Road\\
Oxford, OX1 3TG, UK}
\email{helmut.pitters@stats.ox.ac.uk}
\date{\today}
\begin{abstract}
  We consider both the Bolthausen-Sznitman and the Kingman coalescent restricted to the partitions of $\{1, \ldots, n\}.$ Spectral decompositions of the corresponding generators are derived. As an application we obtain a formula for the Green's functions and a short derivation of the well-known formula for the transition probabilities of the Bolthausen-Sznitman coalescent.
\end{abstract}

\maketitle

\section{Introduction}
An exchangeable coalescent process is a discrete or continuous-time Markov chain that encodes the dynamics of particles grouped into so-called blocks. As time passes, only mergers of some blocks may occur, and the rate at which a merger happens only depends on the current number of blocks, but not, for instance, on their sizes or the specific particles they contain. The theory of exchangeable coalescent processes has its origins in the study of genealogies in population genetics, culminating in the seminal work of Kingman \cite{MR671034}. In the context of population genetics Sagitov \cite{MR1742154} and later Sagitov and M\"{o}hle \cite{MR1880231} derived exchangeable coalescents as limiting genealogies of so-called Cannings models, which are discrete-time models of neutral haploid populations with exchangeable family sizes. This derivation is also implicit in the work of Donnelly and Kurtz \cite{donnelly1999}. Among exchangeable coalescent processes the so-called $\Lambda$-coalescents have received increasing attention in recent years. The latter were introduced independently by Donnelly and Kurtz \cite{donnelly1999}, Pitman \cite{MR1742892} and Sagitov \cite{MR1742154}. A $\Lambda$-coalescent $\{\Pi(t), t\geq 0\}$ is a time-homogeneous exchangeable coalescent process in continuous time with state space $\pset{\mathbb{N}},$ the set of partitions of the non-negative integers $\mathbb{N}\coloneqq \{ 1, 2, \ldots \}$, that only allows for one merger of blocks at any jump. It can be characterized via its restrictions $\{\Pi^n(t), t\geq 0\}$ to $[n]\coloneqq \{1, \ldots, n\}$ as follows. If at any given time $\Pi^n(t)$ contains $b\geq 2$ blocks, then any $2\leq k\leq b$ of these blocks merge at rate $\lambda_{b, k}\coloneqq \int_0^1 x^{k-2}(1-x)^{b-k}\Lambda(dx),$ where $\Lambda$ denotes a finite measure on the unit interval. This measure $\Lambda$ together with the initial state $\Pi(0)$ uniquely determines $\Pi,$ hence the name $\Lambda$-coalescent. 

In this note we consider both the Kingman coalescent $\Pi^K = \{\Pi^K(t), t\geq 0\}$ and the Bolthausen-Sznitman coalescent $\Pi^{BS} = \{\Pi^{BS}(t), t\geq 0\}$, which are both $\Lambda$-coalescents. For convenience we drop the superscripts $K$ and $BS$ when there is no risk of ambiguity. From its introduction Kingman's coalescent has been used in population genetics as a model approximating the genealogy of a sample drawn from a large neutral population of haploid individuals, i.e.~in this context the particles are interpreted as individuals in the population.
The Bolthausen-Sznitman coalescent was discovered by Bolthausen and Sznitman \cite{Bolthausensznitman98} in the context of the Sherrington-Kirkpatrick model for spin glasses in statistical physics. Goldschmidt and Martin \cite{EJP265} gave a construction of $\Pi^{BS}$ via a cutting of a random recursive tree. Bertoin and Le Gall \cite{raey} derived the Bolthausen-Sznitman coalescent as the genealogy of a continuous-state branching process. An increasing number of recent works suggest that the class of beta coalescents, which contains the Bolthausen-Sznitman coalescent, should provide better null models than Kingman's coalescent for the genealogy of highly fecund populations, see the introduction in~\cite{2013arXiv1305.6043B}. Neher and Hallatschek \cite{Neher08012013} argue that the Bolthausen-Sznitman coalescent occurs naturally as the genealogy in models of rapidly adapting asexual populations. Motivated by models of populations undergoing natural selection \cite{BrunetDerrida2006}, \cite{BrunetDerrida2007}, Berestycki, Berestycki and Schweinsberg \cite{berestycki2013} showed that the genealogy of a population governed by Branching Brownian Motion with absorption is again the Bolthausen-Sznitman coalescent.

The article is organized as follows. Our main results are spectral decompositions of the generator of the Bolthausen-Sznitman $n$-coalescent, Theorem \ref{thm:bs_spectral_decomposition}, respectively of the generator of Kingman's $n$-coalescent, Theorem \ref{thm:kingman_spectral_decomposition}. As Corollaries we obtain for the Bolthausen-Sznitman coalescent a derivation of the formula for the transition probabilites that goes back to \cite{Bolthausensznitman98}, and a formula for its Green's matrix. As a further application we obtain a spectral decomposition of the generator of the block counting process of the Bolthausen-Sznitman, respectively Kingman's coalescent.

\section{Results}\label{sec:results}
Let us introduce some notation. A partition of a set $A$ is a set, $\pi$ say, of nonempty pairwise disjoint subsets of $A$ whose union is $A$. The members of $\pi$ are called the blocks of $\pi.$ Let $\card{A}$ denote the cardinality of $A$ and let $\pset{A}$ denote the set of partitions of $A.$
\subsection{Bolthausen-Sznitman $n$-coalescent}
The Bolthausen-Sznitman $n$-coalescent $\Pi^{n, BS}=\{\Pi^{n, BS}(t), t\geq 0\}$ is obtained by choosing $\Lambda$ to be the uniform measure on $[0,1],$ that is the corresponding $Q$-matrix $Q\coloneqq Q^{n, BS}=(q_{\pi\rho})_{\pi,\rho\in\pset{[n]}}$ is given by
\begin{align}
  q_{\pi\rho} = \begin{cases}
    \frac{(\card{\rho}-1)!(\card{\pi}-\card{\rho}-1)!}{(\card{\pi}-1)!} & \text{if } \pi\prec\rho,\\
    -(\card{\pi}-1) & \text{if } \pi=\rho,\\
    0 & \text{otherwise,}
    \end{cases}
  \end{align}
  where $\pi\prec\rho$ if and only if $\rho$ is obtained by exactly one merger of blocks of $\pi.$
  In this section we only consider the Bolthausen-Sznitman $n$-coalescent and therefore write $\Pi$ instead of $\Pi^{n,BS}$.

For any two sets $A$ and $B\subseteq A$ and a partition $\pi\in\pset{A}$ we call $\res{\pi}{B}\coloneqq \{C\cap B\colon C\in\pi, C\cap B\neq\emptyset\}\in\pset{B}$ the restriction of $\pi$ to $B$.
\begin{theorem}[Spectral decomposition of the Bolthausen-Sznitman coalescent]\label{thm:bs_spectral_decomposition}
Let $L=(l_{\pi\rho})_{\pi,\rho\in\pset{[n]}}$ and $R=(r_{\pi\rho})_{\pi,\rho\in\pset{[n]}}$ be matrices defined by
\begin{align}
   l_{\pi\rho} &\coloneqq\begin{cases}
   (-1)^{\card{\pi}-\card{\rho}}\frac{(\card{\rho}-1)!}{(\card{\pi}-1)!} & \text{if } \pi\leq\rho,\\
   0 & \text{otherwise,}
   \end{cases}
   \intertext{and}
   r_{\pi\rho} &\coloneqq\begin{cases}
   \frac{(\card{\rho}-1)!}{(\card{\pi}-1)!}\prod_{B\in\rho}(\card{\res{\pi}{B}}-1)! & \text{if } \pi\leq\rho,\\
   0 & \text{otherwise,}
   \end{cases}
\end{align}
where $\pi\leq\rho$ if and only if each block of $\pi$ is contained in a block of $\rho.$ Then a spectral decomposition of $Q$ is given by $Q=RDL,$ where $D=(d_{\pi\rho})_{\pi,\rho\in\pset{[n]}}$ is defined by $d_{\pi\pi}=-(\card{\pi}-1)$ and $d_{\pi\rho}=0$ if $\pi\neq\rho$. In particular, $l_{\pi\pi}=r_{\pi\pi}=1$ for any $\pi\in\pset{[n]}.$
\end{theorem}
The right eigenvectors $r_{\pi\rho}$ may be interpreted as the probability that a random recursive tree on the label set $\pi$ can be cut down to a tree on the label set $\rho,$ see the paragraph "A connection with random recursive trees" in Section \ref{subsec:bs-coalescent} below. As an application of this spectral decomposition we derive the well-known formula for the transition probabilities of $\Pi$ given by Bolthausen and Sznitman in \cite{Bolthausensznitman98}, Proposition 1.4.
\begin{corollary}[Transition probabilities of the Bolthausen-Sznitman coalescent]\label{cor:bs_transition_probabilities}
  For any two partitions $\pi, \rho\in\pset{[n]}$ and any time $t\geq 0$ the transition probabilities $p_{\pi\rho}(t)\coloneqq \Prob{\Pi(t)=\rho|\Pi(0)=\pi}$ of the Bolthausen-Sznitman $n$-coalescent are given by
  \begin{align*}
    p_{\pi\rho}(t) &= (-1)^{\card{\rho}}e^t\frac{(\card{\rho}-1)!}{(\card{\pi}-1)!} 
\prod_{B\in\rho} (-e^{-t})^{\overline{\card{\res{\pi}{B}}}},
  \end{align*}
  where for $x\in\mathbb{R},$ $k\in\mathbb{N}$ we denote by $x^{\overline{k}}\coloneqq x(x+1)\cdots (x+k-1)$ the ascending factorial power with the convention $x^{\overline{0}}\coloneqq 1$.
\end{corollary}
Thanks to the spectral decomposition, Theorem \ref{thm:bs_spectral_decomposition}, we obtain a formula for the Green's matrix $G=(g_{\pi\rho})_{\pi,\rho\in\pset{[n]}}$ of the Bolthausen-Sznitman $n$-coalescent defined by $g_{\pi\rho}\coloneqq \int_0^\infty p_{\pi\rho}(t)dt$. Recall that $g_{\pi\rho}=\Ex[\int_0^{\infty} 1_{\{\Pi(t)=\rho\}}dt|\Pi(0)=\pi]$ is the expected total time that $\Pi$ spends in $\rho$ starting from $\pi$. We denote by ${i\brack j}$ the Stirling permutation numbers or unsigned Stirling numbers of the first kind, which count the number of permutations of a set of $i$ elements with $j$ cycles.
\begin{corollary}[Green's matrix of the Bolthausen-Sznitman coalescent]\label{cor:bs_greens_matrix}
  The Green's matrix $G=(g_{\pi\rho})_{\pi,\rho\in\pset{[n]}}$ is given by
  \begin{align}
    g_{\pi\rho} &= \begin{cases}
      (-1)^{\card{\rho}}\frac{(\card{\rho}-1)!}{(\card{\pi}-1)!}\sum_{(k_B)_{B\in\rho}\in\mathbb{N}^{\card{\rho}}}\frac{(-1)^{\abs k}}{\abs k-1}\prod_{B\in\rho}{\card{\res{\pi}{B}}\brack k_B} & \text{if } \pi\leq\rho\neq \{[n]\},\\
      \infty & \text{if }\pi\leq\rho=\{[n]\},\\
      0 & \text{otherwise},
      \end{cases}
  \end{align}
  where $\abs k\coloneqq\sum_{B\in\rho} k_B.$
\end{corollary}
\begin{remark}
Notice that since the return probability to any state $\pi\in\pset{[n]},$ $\pi\neq \{[n]\},$ equals $0$ for any $\Lambda$-coalescent $\Pi$ there is a close connection between the Green's matrix $G$ of $\Pi$ and its hitting probabilities defined by
  \[h(\pi, \rho)\coloneqq \Prob{\Pi \text{ hits }\rho\text{ when started from }\pi},\]
namely via $g_{\pi\rho}=h(\pi, \rho)/q_{\rho}=h(\pi, \rho)/(1-\card{\rho}),$ cf.~\cite{DBLP:books/daglib/0095301}, p.~146, where $q_\rho\coloneqq \sum_{\sigma\in\pset{[n]}, \sigma\neq\rho}q_{\rho\sigma}$ is the total rate in $\rho$.
\end{remark}

For a set $A$ and $j\in\mathbb{N}$ let $\pset{A, j}$ denote the set of partitions of $A$ into $j$ blocks. Moreover, ${i\brace j}$ denotes the Stirling partition numbers or Stirling numbers of the second kind, which count the number of partitions into $j$ blocks of a set of $i$ elements.
\begin{remark}
From Corollary \ref{cor:bs_greens_matrix} it follows by a technical but straightforward computation that for any $\pi\in\pset{[n], i}$ and $j\le i$
\begin{align*}
\sum_{\rho\in\pset{[n], j}} g_{\pi\rho} = (-1)^{j} \frac{(j-1)!}{(i-1)!} \sum_{k=j}^i \frac{(-1)^k}{k-1}{i\brack k}{k\brace j}
\end{align*}
in agreement with the entries of the Green's matrix of the block counting process of the Bolthausen-Sznitman coalescent provided in the proof of Corollary 1.5 in \cite{ECP3464}. 
\end{remark}

As a further application of the spectral decomposition of $Q,$ Theorem \ref{thm:bs_spectral_decomposition}, we derive a spectral decomposition of the generator $Q'= (q'_{ij})_{i, j\in [n]}$ of the block counting process $\{N(t), t\geq 0\}$ of the Bolthausen-Sznitman $n$-coalescent defined by $N(t)\coloneqq \card{\Pi(t)}.$ It is well-known that the matrix $Q'$ is given by
\begin{align*}
q'_{ij}= \begin{cases}
  \frac{i}{(i-j)(i-j+1)} & \text{if } i>j,\\
  1-i & \text{if } i=j,\\
  0 & \text{otherwise.}
\end{cases}
\end{align*}
In \cite{ECP3464} this spectral decomposition of $Q'$ was derived by means of generating functions without recourse to the partition-valued process $\Pi$.
\begin{corollary}\label{cor:bs_spectral_decomposition_block_counting_proces}
Let $L'=(l'_{ij})_{i,j\in [n]},$ $R'=(r'_{ij})_{i,j\in [n]},$ and $D'=(d'_{ij})_{i,j\in [n]}$ be matrices given by
  \begin{align}
    l'_{ij}\coloneqq (-1)^{i-j}\frac{(j-1)!}{(i-1)!}{i\brace j}, \quad r'_{ij}\coloneqq \frac{(j-1)!}{(i-1)!}{i\brack j},\quad d'_{ij}\coloneqq (1-i)1_{\{i=j\}},
  \end{align}
  Then a spectral decomposition of $Q'$ is given by $Q'=R'D'L'.$
\end{corollary}

\subsection{Kingman's $n$-coalescent}
Kingman's $n$-coalescent $\Pi^{n,K}=\{\Pi^{n,K}(t), t\geq 0\}$ is obtained by choosing $\Lambda$ to be $\delta_0,$ the Dirac measure in $0,$ that is the corresponding $Q$-matrix $Q\coloneqq Q^{n,K}=(q_{\pi\rho})_{\pi,\rho\in\psetn}$ is given by
\begin{align}
  q_{\pi\rho} = \begin{cases}
    1 & \text{if } \pi\lessdot\rho,\\
    -{\card{\pi}\choose 2} & \text{if } \pi=\rho,\\
    0 & \text{otherwise},
    \end{cases}
\end{align}
where $\pi\lessdot\rho$ if and only if $\pi\leq\rho$ and $\card{\pi}-\card{\rho}=1.$ In other words, the jump chain of Kingman's $n$-coalescent is the directed simple random walk on the partition lattice $\pset{[n]}$, where at each step the chain jumps into a coarser partition. From now on we only consider Kingman's $n$-coalescent and therefore write $\Pi$ instead of $\Pi^{n,K}$.
\begin{theorem}[Spectral decomposition of Kingman's coalescent]\label{thm:kingman_spectral_decomposition}
Define the matrices $L=(l_{\pi\rho})_{\pi,\rho\in\pset{[n]}}$ and $R=(r_{\pi\rho})_{\pi,\rho\in\pset{[n]}}$ by
\begin{align}
l_{\pi\rho} &\coloneqq \begin{cases}
 (-1)^{\card{\pi}-\card{\rho}}\frac{(\card{\pi}+\card{\rho}-2)!}{(2\card{\pi}-2)!}\prod_{B\in\rho}\card{\res{\pi}{B}}! & \text{if }\pi\leq\rho,\\
 0 & \text{otherwise,}
\end{cases}
\intertext{and}
    r_{\pi\rho} &\coloneqq
     \begin{cases}
      \frac{(2\card{\rho}-1)!}{(\card{\pi}+\card{\rho}-1)!}\prod_{B\in\rho}\card{\res{\pi}{B}}! & \text{if } \pi\leq\rho,\\
      0 & \text{otherwise}.
    \end{cases}
\end{align}
Then a spectral decomposition of $Q$ is given by $Q=RDL,$ where $D=(d_{\pi\rho})_{\pi,\rho\in\pset{[n]}}$ is defined by $d_{\pi\pi}\coloneqq -\binom{\card{\pi}}{2}$ and $d_{\pi\rho}=0$ if $\pi\neq\rho$. In particular, $l_{\pi\pi}=r_{\pi\pi}=1$ for any $\pi\in\pset{[n]}.$
\end{theorem}

\begin{remark}
Although the spectral decomposition of $Q,$ Theorem \ref{thm:kingman_spectral_decomposition}, directly yields a spectral decomposition of the matrix of transition probabilites $p_{\pi\rho}(t)\coloneqq \Prob{\Pi(t)=\rho|\Pi(0)=\pi},$ analoguously to the one given in \eqref{eq:bs_transition_probabilities} for the Bolthausen-Sznitman $n$-coalescent, this expression is not particularly handy. In fact, a better approach to calculating $p_{\pi\rho}(t)$ is the one given by Kingman in \cite{MR671034} Equation (2.5).
\end{remark}

As in the case of the Bolthausen-Sznitman coalescent, consider the block counting process $\{N(t), t\geq 0\}$ of Kingman's $n$-coalescent defined by $N(t)\coloneqq \card{\Pi(t)}$ and let $Q'=(q'_{ij})_{i,j\in [n]}$ denote the generator of $N(t)$. From Theorem \ref{thm:kingman_spectral_decomposition} we obtain a spectral decomposition of $Q'$.

\begin{corollary}\label{cor:kingman_spectral_decomposition_block_counting_proces}
  Let $L'=(l'_{ij})_{i,j\in [n]},$ $R'=(r'_{ij})_{i,j\in [n]}$ be matrices defined by
  \begin{align}
    l'_{ij}\coloneqq (-1)^{i+k}\frac{(i+j-2)!}{(2i-2)!}\lah{i}{j},\qquad r'_{ij}\coloneqq \frac{(2j-1)!}{(i+j-1)!}\lah{i}{j},
  \end{align}
where $\lah{i}{j}\coloneqq\binom{i-1}{j-1}\frac{i!}{j!}$ denotes the unsigned Lah numbers which count the number of partitions into $j$ blocks of a set of $i$ elements, where the elements in each block are ordered. Then $Q'=R'D'L'$ is a spectral decomposition of $Q',$ where $D'=(d'_{ij})_{i,j\in\pset{[n]}}$ is defined by $d'_{ii}=-\binom{i}{2}$ and $d'_{ij}=0$ if $i\neq j.$
\end{corollary}

\section{Proofs}
In order to prove our results, we need some notions and facts from the theory of lattices, that we collect from \cite{Stanley:2011:ECV:2124415}. Recall that a partially ordered set (poset for short) $(P, \leq)$ is a set $P$ together with a binary relation $\leq$ satisfying:
\begin{enumerate}
\item For all $p\in P,$ $p\leq p$ (reflexivity).
\item If $p\leq q$ and $q\leq p,$ then $p=q$ (antisymmetry).
\item If $p\leq q$ and $q\leq r,$ then $p\leq r$ (transitivity).
\end{enumerate}
Recall that two posets $P, Q$ are called isomorphic, in which case we write $P\cong Q,$ if there exists an order-preserving bijection $\phi\colon P\to Q$ whose inverse is order-preserving. The cartesian product $P\times Q$ of two posets $P, Q$ is defined on the set $\{(p, q)\colon p\in P, q\in Q\}$ by letting $(p,q)\leq(p',q')$ in $P\times Q$ if and only if $p\leq p'$ in $P$ and $q\leq q'$ in $Q$.
For $p, q\in P$ an upper bound of $p$ and $q$ is an element $r\in P$ such that $p,q\leq r$. A least upper bound of $p$ and $q$ is an upper bound $s\in P$ of $p$ and $q$ such that for any upper bound $r$ of $p$ and $q$ one has $s\leq r$. Clearly, if a least upper bound of two elements $p$ and $q$ exists, it is unique. A greatest lower bound is defined in complete analogy. A lattice $L$ is a poset with the property that any two of its elements have a least upper bound and a greatest lower bound. It is well-known, that $\pset{[n]}$ together with the relation $\leq$ as defined in Theorem \ref{thm:bs_spectral_decomposition} is a lattice, the so-called partition lattice. For $\pi,\rho\in\pset{[n]}$ with $\pi\leq\rho$ we call the set $[\pi, \rho]\coloneqq\{\sigma\in\pset{[n]}\colon \pi\leq\sigma\leq\rho\}$ an interval. We will make repeated use of the isomorphism
  \begin{align}
    [\pi, \rho]\cong \bigtimes_{B\in\rho} \pset{\res{\pi}{B}},
  \end{align}
  cf.~Example 3.10.4 in \cite{Stanley:2011:ECV:2124415}. For more information on posets in general and the partition lattice in particular the reader is referred to \cite{Stanley:2011:ECV:2124415}. Evidently, we have $\pset{[n]}=[\Delta_{[n]}, \{[n]\}],$ where for any set $A$ we let $\Delta_A\coloneqq \{\{a\}\colon a\in A\}$ be the partition of $A$ into singletons.
  
  Using the notation we just introduced, the $n$-$\Lambda$-coalescent $\Pi^n$ is a Markov chain with state space $\pset{[n]},$ initial state $\Delta_{[n]}$ and $Q$-matrix $Q^{n}=(q^{n}_{\pi\rho})_{\pi,\rho\in\pset{[n]}}$ given by
  \begin{align}
  q^{n}_{\pi\rho} \coloneqq \begin{cases}
   \lambda_{\card{\pi}, \card{\pi}-\card{\rho}+1} & \text{if } \pi\prec\rho,\\
   -\lambda_{\card{\pi}} & \text{if }\pi=\rho,\\
   0 & \text{otherwise,}
  \end{cases}
  \end{align}
  where $\lambda_b\coloneqq\sum_{k=2}^b \binom{b}{k}\lambda_{b,k},$ $b\geq 2,$ is the infinitesimal rate.
There are several extensions of $\leq$ to a linear order on $\pset{[n]}$. Let us fix such an extension $\leq_{\text{ex}}$ and notice that the following quantities of $Q^n$ that we are interested in do not depend on the specific extension chosen. The linear order $\leq_{\text{ex}}$ induces a natural bijection $\psi$ from $\pset{[n]}$ to $[B_n],$ where $B_n$ denotes the $n$th Bell number, which is the number $\card{\pset{[n]}}$ of partitions of $[n],$ defined inductively by letting $\psi(\Delta_{[n]})=1$ and $\psi(\pi)\leq\psi(\rho)$ iff $\pi\leq_{\text{ex}}\rho$. Then $Q^n$ can be seen as an upper right triangular matrix with entries ordered according to $\leq_{\text{ex}},$ if we define row/column $\pi$ to be lower than row/column $\rho$ iff $\pi\leq_{\text{ex}}\rho$. The determinant of $Q^n$ is therefore given by the product of its diagonal entries. Hence, the characteristic polynomial of $Q^{n}$ is given by
  \begin{align*}
    \chi_{Q^n}(x) & \coloneqq \det (Q^n-x\mathbb{I}_n) = (-1)^{B_n}\prod_{\pi\in\psetn} (\lambda_{\card{\pi}}+x)
    = (-1)^{B_n}\prod_{i=1}^n (\lambda_i+x)^{{n\brace i}},
  \end{align*}
  where $\mathbb{I}_n$ is the identity matrix on $\psetn$. Hence, for each $i\in [n],$ $-\lambda_i$ is an eigenvalue of $Q^{n}$ with algebraic multiplicity $\card{\mathcal{P}_{[n], i}} = {n\brace i}$. Since $Q^n$ is a $Q$-matrix, $(1, \ldots, 1)^{\top}\in\mathbb{R}^{B_n}$ is an eigenvector corresponding to the eigenvalue $0,$ and direct inspection yields that $(1, 0, \ldots, 0)^{\top}\in\mathbb{R}^{B_n}$ is also an eigenvector corresponding to the eigenvalue $-\lambda_n$.
  From now on we fix an $n\in\mathbb{N},$ $n\geq 2,$ and drop this index in the notation, if there is no risk of confusion.
\subsection{Bolthausen-Sznitman $n$-coalescent}\label{subsec:bs-coalescent}
In order to prepare the proof of the spectral decomposition of $Q=Q^{n, BS},$ Theorem \ref{thm:bs_spectral_decomposition}, we calculate the left and right eigenvectors of $Q$. In the sequel we give two proofs for the right eigenvectors of $Q$ that are of rather different flavours. The first proof is completeley self-contained and only makes use of the partition lattice $\pset{[n]}$. Together with the proof of Lemma \ref{lemma:kingman_right_eigenvector} it might serve as a starting point to find a spectral decomposition for more general coalescents, e.g.~beta coalescents. There is a probabilistic interpretation of the right eigenvector of $Q$ in terms of random recursive trees which then motivates our second proof that heavily draws on random recursive trees and their connection to the Bolthausen-Sznitman coalescent as explored in Goldschmidt and Martin \cite{EJP265}.

\begin{lemma}\label{lemma:bs_right_eigenvector}
For $\rho\in\pset{[n]}$ the vector $(r_{\pi\rho})_{\pi\in\pset{[n]}}$ defined by
 \begin{align}
   r_{\pi\rho}\coloneqq\begin{cases}
   \frac{(\card{\rho}-1)!}{(\card{\pi}-1)!}\prod_{B\in\rho}(\card{\res{\pi}{B}}-1)! & \text{if } \pi\leq\rho,\\
   0 & \text{otherwise,}
   \end{cases}
 \end{align}
is a right eigenvector of $Q$ with corresponding eigenvalue $1-\card{\rho}.$
\end{lemma}
\begin{proof}(first proof of Lemma \ref{lemma:bs_right_eigenvector})
In order to carry out the following calculations, for any two partitions $\pi, \rho\in\pset{[n]}$ we need to parameterize the set $\{\sigma\colon \pi\prec\sigma\leq\rho\}.$ To construct an arbitrary partition $\sigma$ such that $\pi\prec\sigma,$ we could choose a subset $C\subseteq\pi$ of at least two blocks of $\pi$ and merge them in order to obtain $\sigma$. In this case $\card{\sigma}=\card{\pi}-\card{C}+1$. If, additionally, we require $\sigma\leq\rho,$ we certainly cannot choose any collection $C$ of blocks in $\pi.$ Instead, all blocks chosen have to be in $\res{\pi}{B}$ for some block $B\in\rho,$ in which case $\card{\res{\sigma}{B}}=\card{\res{\pi}{B}}-\card{C}+1$. To summarize, we have
\begin{align*}
  \{\sigma\colon \pi\prec\sigma\leq\rho\} &= \left\{\left\{\bigcup_{D\in C} D\right\} \cup\left(\pi\setminus C\right)\colon B\in\rho, C\subseteq\res{\pi}{B}, \card{C}\geq 2\right\}.
\end{align*}
Using this parametrization, we obtain
\begin{align*}
  &\sum_{\sigma\in\pset{[n]}} q_{\pi\sigma}r_{\sigma\rho} \\
  &= \sum_{\sigma\colon \pi\prec\sigma\leq\rho} q_{\pi\sigma}r_{\sigma\rho} - (\card{\pi}-1)r_{\pi\rho}\\
  &= \sum_{B\in\rho} \sum_{\substack{C\subseteq\res{\pi}{B}\\\card{C}\geq 2}} \frac{(\card{\pi}-\card{C})!(\card{C}-2)!}{(\card{\pi}-1)!}\frac{(\card{\rho}-1)!}{(\card{\pi}-\card{C})!}\\
    &\qquad\times (\card{\res{\pi}{B}}-\card{C})!\prod_{D\in\rho\setminus \{B\}}(\card{\res{\pi}{D}}-1)!- (\card{\pi}-1)r_{\pi\rho}\\
  &= \sum_{B\in\rho} \sum_{c=2}^{\card{\res{\pi}{B}}}\binom{\card{\res{\pi}{B}}}{c}\frac{(c-2)!}{(\card{\pi}-1)!}(\card{\rho}-1)!\frac{(\card{\res{\pi}{B}}-c)!}{(\card{\res{\pi}{B}}-1)!}\prod_{D\in\rho} (\card{\res{\pi}{D}}-1)! - (\card{\pi}-1)r_{\pi\rho}\\
  &= \left(\sum_{B\in\rho} \sum_{c=2}^{\card{\res{\pi}{B}}}\binom{\card{\res{\pi}{B}}}{c} (c-2)!\frac{(\card{\res{\pi}{B}}-c)!}{(\card{\res{\pi}{B}}-1)!} - (\card{\pi}-1)\right) r_{\pi\rho}\\
  &= \left(\sum_{B\in\rho} \card{\res{\pi}{B}}\sum_{c=2}^{\card{\res{\pi}{B}}} \frac{1}{c(c-1)}  - (\card{\pi}-1)\right) r_{\pi\rho}\\
  &= \left(\sum_{B\in\rho} (\card{\res{\pi}{B}}-1)  - (\card{\pi}-1)\right) r_{\pi\rho}= (1-\card{\rho})r_{\pi\rho},
\end{align*}
and the claim follows.
\end{proof}

{\bf A connection with random recursive trees.} Let us now recall the notion of a random recursive tree in order to prepare our second proof of Lemma \ref{lemma:bs_right_eigenvector}, where we closely follow Goldschmidt and Martin \cite{EJP265}. We call a tree on $n$ nodes labelled by $1, 2, \ldots, n$ an increasing tree if the root has label $1$ and the labels in any path from the root to another node are increasing. Figure \ref{fig:increasing_trees} shows all $6$ increasing trees on $\{\{1\}, \{2\}, \{3\}, \{4\}\}.$
  \begin{figure}[ht]
    \begin{center}
          \begin{tikzpicture}[scale=0.75, auto]
            \node[draw, fill=black, circle, minimum size=5pt, inner sep=0pt, label=below:$\{1\}$] () at (0, 0) {};
            \node[draw, fill=black, circle, minimum size=5pt, inner sep=0pt, label=right:$\{2\}$] () at (0, 1) {};
            \node[draw, fill=black, circle, minimum size=5pt, inner sep=0pt, label=right:$\{3\}$] () at (0, 2) {};
            \node[draw, fill=black, circle, minimum size=5pt, inner sep=0pt, label=above:$\{4\}$] () at (0, 3) {};
            \draw (0, 0) -- (0, 1) -- (0, 2) -- (0, 3);
          \end{tikzpicture}
          \begin{tikzpicture}[scale=0.75, auto]
            \node[draw, fill=black, circle, minimum size=5pt, inner sep=0pt, label=below:$\{1\}$] () at (0, 0) {};
            \node[draw, fill=black, circle, minimum size=5pt, inner sep=0pt, label=right:$\{2\}$] () at (0, 1) {};
            \node[draw, fill=black, circle, minimum size=5pt, inner sep=0pt, label=above:$\{3\}$] () at (-1, 2) {};
            \node[draw, fill=black, circle, minimum size=5pt, inner sep=0pt, label=above:$\{4\}$] () at (1, 2) {};
            \draw (0, 0) -- (0, 1) -- (-1, 2);
            \draw (0,1) -- (1, 2);
          \end{tikzpicture}
            \begin{tikzpicture}[scale=0.75, auto]
              \node[draw, fill=black, circle, minimum size=5pt, inner sep=0pt, label=below:$\{1\}$] () at (0, 0) {};
              \node[draw, fill=black, circle, minimum size=5pt, inner sep=0pt, label=right:$\{2\}$] () at (-1, 1) {};
              \node[draw, fill=black, circle, minimum size=5pt, inner sep=0pt, label=above:$\{3\}$] () at (-1, 2) {};
              \node[draw, fill=black, circle, minimum size=5pt, inner sep=0pt, label=above:$\{4\}$] () at (1, 1) {};
              \draw (1, 1) -- (0, 0) -- (-1, 1) -- (-1, 2);
            \end{tikzpicture}  
      \begin{tikzpicture}[scale=0.75, auto]
        \node[draw, fill=black, circle, minimum size=5pt, inner sep=0pt, label=below:$\{1\}$] () at (0, 0) {};
        \node[draw, fill=black, circle, minimum size=5pt, inner sep=0pt, label=right:$\{2\}$] () at (-1, 1) {};
        \node[draw, fill=black, circle, minimum size=5pt, inner sep=0pt, label=above:$\{4\}$] () at (-1, 2) {};
        \node[draw, fill=black, circle, minimum size=5pt, inner sep=0pt, label=above:$\{3\}$] () at (1, 1) {};
        \draw (1, 1) -- (0, 0) -- (-1, 1) -- (-1, 2);
      \end{tikzpicture}
      \begin{tikzpicture}[scale=0.75, auto]
              \node[draw, fill=black, circle, minimum size=5pt, inner sep=0pt, label=below:$\{1\}$] () at (0, 0) {};
              \node[draw, fill=black, circle, minimum size=5pt, inner sep=0pt, label=above:$\{2\}$] () at (-1, 1) {};
              \node[draw, fill=black, circle, minimum size=5pt, inner sep=0pt, label=above:$\{4\}$] () at (1, 2) {};
              \node[draw, fill=black, circle, minimum size=5pt, inner sep=0pt, label=right:$\{3\}$] () at (1, 1) {};
              \draw (1, 2) -- (1, 1) -- (0, 0) -- (-1, 1);
      \end{tikzpicture}
      \begin{tikzpicture}[scale=0.75, auto]
              \node[draw, fill=black, circle, minimum size=5pt, inner sep=0pt, label=below:$\{1\}$] () at (0, 0) {};
              \node[draw, fill=black, circle, minimum size=5pt, inner sep=0pt, label=above:$\{2\}$] () at (-1, 1) {};
              \node[draw, fill=black, circle, minimum size=5pt, inner sep=0pt, label=above:$\{3\}$] () at (0, 1) {};
              \node[draw, fill=black, circle, minimum size=5pt, inner sep=0pt, label=above:$\{4\}$] () at (1, 1) {};
              \draw (1, 1) -- (0, 0) -- (-1, 1);
              \draw (0, 0) -- (0, 1);
      \end{tikzpicture}
    \end{center}
    \caption{All $6$ increasing trees on $\{\{1\}, \{2\}, \{3\}, \{4\}\}.$ Only the third and the sixth tree contain $\{\{1, 2, 3\}, \{4\}\}.$}
    \label{fig:increasing_trees}
  \end{figure}
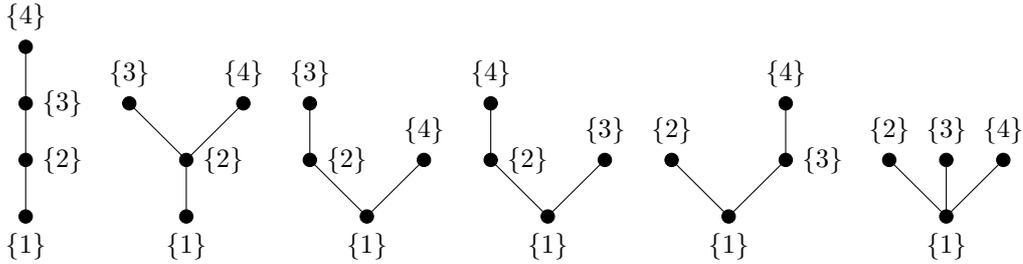
  If we have an increasing tree on $n-1$ nodes, we obtain an increasing tree on $n$ nodes by adding a node labelled $n$ and attaching it by an edge to one of the nodes in the given tree. Starting from the tree that only consists of the root node $1$ this gives an explicit construction of all increasing trees on $n$ nodes. Consequently, there are $(n-1)!$ increasing trees on $n$ nodes. A random recursive tree is a tree chosen uniformly at random from all increasing trees on $n$ nodes. An explicit construction of a random recursive tree on $n$ nodes is the following. Start with the root node labelled $1$. If the tree has $k$ nodes, choose one of these nodes uniformly at random and attach to it node $k+1$ by an edge. Stop after attaching node $n$.

For any partition $\pi\in\pset{[n]}$ an increasing tree on $\pi$ is a tree with $\card\pi$ nodes that are labelled by the blocks in $\pi$ such that the labels in any path from the root to another node are increasing with respect to their least element. We denote a random recursive tree on $\pi$ by $\mathfrak T_\pi.$

Crucial to the construction of the Bolthausen-Sznitman coalescent via random recursive trees is the following cutting procedure. When given a tree $\mathcal T$ on $\pi,$ a cut is performed by picking an edge, removing the subtree above this edge (here we picture trees as they grow in nature: from the root at the bottom to the leaves at the top) and adding the labels of this subtree to the labels of the node below the edge. For a simple tree the cutting procedure is depicted in Figure \ref{fig:cutting_trees}.
  \begin{figure}[ht]
    \begin{center}
      \begin{tikzpicture}[scale=0.75, auto]
        \node[draw, fill=black, circle, minimum size=5pt, inner sep=0pt, label=below:$\{1\}$] () at (0, 0) {};
        \node[draw, fill=black, circle, minimum size=5pt, inner sep=0pt, label=right:$\{2\}$] () at (-1, 1) {};
        \node[draw, fill=black, circle, minimum size=5pt, inner sep=0pt, label=above:$\{4\}$] () at (-1, 2) {};
        \node[draw, fill=black, circle, minimum size=5pt, inner sep=0pt, label=right:$\{3\}$] () at (1, 1) {};
		\node[draw, cross out, minimum width=5pt, fill=red] () at (-0.5, 0.5) {};
        \draw (1, 1) -- (0, 0) -- (-1, 1) -- (-1, 2);
      \end{tikzpicture}
      \begin{tikzpicture}[scale=0.75, auto]
        \node[draw, fill=black, circle, minimum size=5pt, inner sep=0pt, label=below:$\{{1, 2, 4}\}$] () at (0, 0) {};
        \node[draw, fill=black, circle, minimum size=5pt, inner sep=0pt, label=right:$\{3\}$] () at (1, 1) {};
		\node[draw, cross out, minimum width=5pt, fill=red] () at (0.5, 0.5) {};
        \draw (1, 1) -- (0, 0);
      \end{tikzpicture}
      \begin{tikzpicture}[scale=0.75, auto]
        \node[draw, fill=black, circle, minimum size=5pt, inner sep=0pt, label=below:$\{{1, 2, 3, 4}\}$] () at (0, 0) {};
      \end{tikzpicture}
    \end{center}
    \caption{An increasing tree on $\{\{1\}, \{2\}, \{3\}, \{4\}\}$ cut down successively to a tree on one node $\{1, 2, 3, 4\}.$}
    \label{fig:cutting_trees}
  \end{figure}
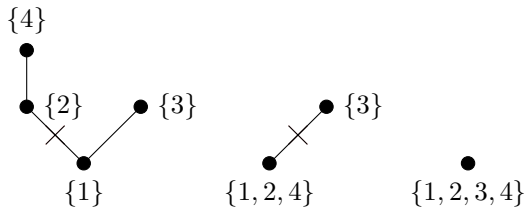
We denote by $c\mathcal T$ the tree obtained by cutting $\mathcal T$ at an edge chosen uniformly at random. A striking result then is Proposition 2.1 in Goldschmidt and Martin \cite{EJP265}, which states that after cutting a random recursive tree $\mathfrak T_\pi$ on $\pi$ at an edge chosen uniformly at random, the new tree $c\mathfrak T_\pi$ is a random recursive tree on the new label set (which is again a partition of $[n]$).

For $\pi\in\pset{[n]}$ define a time-homogeneous continuous-time Markov chain $\mathcal R_\pi \coloneqq \{ \mathcal R_\pi(t), t\geq 0 \}$ with values in the set of random recursive trees as follows. The initial state $\mathcal R_\pi(0)$ is the random recursive tree $\mathfrak T_\pi$ on $\pi.$ The process $\mathcal R_\pi$ evolves according to the following dynamics. Suppose $\mathcal R_\pi$ is in state $\mathcal T.$ If $\mathcal T$ has only one vertex, do nothing. Otherwise, attach to each edge in $\mathcal T$ an exponential $1$ clock, all clocks being independent. When the first clock rings, cut $\mathcal T$ at the associated edge to obtain the next state of $\mathcal R_\pi.$ For any tree $\mathcal T$ on the label set $\pi$ let $p(\mathcal T)=\pi.$ Proposition 2.2 of Goldschmidt and Martin \cite{EJP265} establishes that the Bolthausen-Sznitman $n$-coalescent $\Pi$ is equal in distribution to
\[ \{ p(\mathcal R_{\Delta_n} (t)), t\geq 0  \}. \]

Fix two partitions $\pi, \rho\in\pset{[n]}$ such that $\pi\leq\rho$. We say that an increasing tree $\mathcal T$ on $\pi$ contains $\rho$ iff one can obtain an increasing tree on $\rho$ by successively cutting $\mathcal T.$ If $\mathcal T$ contains $\rho$ we write $\rho\sqsubset \mathcal T$. Let $\{R_\pi(k), k\geq 0\}$ denote the jump chain of $\mathcal R_\pi$. From the definition of $\mathcal R$ it follows that a process that is equal in distribution to the jump chain $R$ can be constructed recursively by letting
\[R_\pi(0) \coloneqq \mathfrak T_\pi, \quad R_\pi(k+1)\coloneqq cR_\pi(k), \quad k\geq 0.\]
A simple question then is: what is the probability $\Prob{\rho\sqsubset \mathfrak T_\pi}$ that a random recursive tree on $\pi$ contains $\rho$? Notice first, that an increasing tree $\mathcal T$ on $\pi$ contains $\rho$ if for each block $B\in\rho$ we can find a node $v=v(B)$ in $\mathcal T$ such that the labels in the subtree above $v$ coincide with the elements in $\res{\pi}{B}$. In other words, we can construct all increasing trees on $\pi$ that contain $\rho$ by first constructing an increasing tree $\mathcal T$ on $\rho.$ Then each node $v\in\mathcal T$ is labelled by some block $B\in\rho$. Now if $\card{\res{\pi}{B}}>1,$ build an increasing tree on $\res{\pi}{B}$ and replace $v$ by this tree. This procedure is done for all nodes $v\in\mathcal T$ to obtain an increasing tree on $\pi$ that contains $\rho.$ Therefore, the number of increasing trees on $\pi$ containing $\rho$ is
\begin{align*}
  \card \{\mathcal T\colon \mathcal T\text{ increasing tree on }\pi, \rho\sqsubset \mathcal T\} = (\card\rho-1)!\prod_{B\in\rho}(\card\res{\pi}{B}-1)!.
\end{align*}
On the other hand, the total number of increasing trees on $\pi$ is $(\card\pi-1)!$. Since a random recursive tree on $\pi$ is a tree chosen uniformly at random from all increasing trees on $\pi,$ we obtain
\begin{align}
\Prob{\rho\sqsubset \mathfrak T_\pi} = \begin{cases}
\frac{(\card\rho-1)!\prod_{B\in\rho}(\card\res{\pi}{B}-1)!}{(\card\pi-1)!} & \text{if }\pi\leq\rho,\\
0 & \text{otherwise},
\end{cases}
\end{align}
which is just $r_{\pi\rho}$. We can now turn to our second proof of Lemma \ref{lemma:bs_right_eigenvector} in terms of random recursive trees.
\begin{proof}\label{proof:second_proof}(second proof of Lemma 1)
We rewrite the statement as
\begin{align}\label{eq:statement_rewritten}
  \sum_{\sigma\colon \pi\prec\sigma} q_{\pi\sigma}r_{\sigma\rho} &= (\card\pi-\card\rho)r_{\pi\rho}.
\end{align}
Let $J=\{J(k), k\geq 0\}$ denote the jump chain of $\Pi=\Pi^{BS, n}$. Since $\Prob{J(1)=\sigma|J(0)=\pi}=q_{\pi\sigma}/q_\pi$ is the probability that $J$ jumps from $\pi$ to $\sigma,$ after dividing \eqref{eq:statement_rewritten} by $q_\pi=\card\pi-1$ we obtain for the left hand side
\begin{align*}
  \sum_{\sigma\colon \pi\prec\sigma} \frac{q_{\pi\sigma}}{q_\pi}r_{\sigma\rho} & = \sum_{\sigma\colon \pi\prec\sigma} \Prob{J(1)=\sigma|J(0)=\pi}\Prob{\rho\sqsubset \mathfrak T_\sigma}\\
  &= \sum_{\sigma\colon \pi\prec\sigma} \Prob{p(c\mathfrak T_\pi)=\sigma}\Prob{\rho\sqsubset \mathfrak T_\sigma}\\
  &= \Prob{\rho\sqsubset \mathfrak T_{p(c\mathfrak T_\pi)}} \\
  &= \Prob{\rho\sqsubset c\mathfrak T_\pi} \\
  &= \Prob{\rho\sqsubset c\mathfrak T_\pi|\rho\sqsubset \mathfrak T_\pi}\Prob{\rho\sqsubset \mathfrak T_\pi}, \\
\end{align*}
where we used that for any random recursive tree $\mathcal T$ one has $\mathfrak T_{p(\mathcal T)}=_d\mathcal T$.
Conditional on $\rho\sqsubset \mathfrak T_\pi,$ there are precisely $\card\pi-\card\rho$ among the $\card\pi-1$ edges in $\mathfrak T_\pi$ that we may cut in order to obtain a tree $c\mathfrak T_\pi$ that contains $\rho.$ Since, by definition, $c\mathfrak T_\pi$ is obtained by cutting $\mathfrak T_\pi$ at an edge chosen uniformly at random, we have
\begin{align*}
  \Prob{\rho\sqsubset c\mathfrak T_\pi|\rho\sqsubset \mathfrak T_\pi}=\frac{\card{\pi-\card{\rho}}}{\card \pi-1}.
\end{align*}
The claim follows.
\end{proof}

\begin{lemma}\label{lem:technical_lemma}
For $x\in\mathbb{R}\setminus\{0\}$ we have
\begin{align}
	\sum_{\sigma\in\pset{[n]}} r_{\pi\sigma} x^{\card{\sigma}-1}l_{\sigma\rho} &= (-1)^{\card{\rho}}x^{-1}\frac{(\card{\rho}-1)!}{(\card{\pi}-1)!} \prod_{B\in\rho}(-x)^{\overline{\card{\res{\pi}{B}}}}.
\end{align}
\end{lemma}
\begin{proof}
Notice that
  \begin{align}\label{eq:bs_transition_matrix}
\sum_{\sigma\in\pset{[n]}} r_{\pi\sigma} x^{\card{\sigma}-1}l_{\sigma\rho}\notag
    &= \frac{(\card{\rho}-1)!}{(\card{\pi}-1)!}\sum_{\sigma\in [\pi,\rho]} (-1)^{\card{\sigma}-\card{\rho}} x^{\card{\sigma}-1}\prod_{B\in\sigma}(\card{\res{\pi}{B}}-1)!\notag\\
         &=(-1)^{\card{\rho}}x^{-1}\frac{(\card{\rho}-1)!}{(\card{\pi}-1)!}  \sum_{\sigma\in [\pi,\rho]} \prod_{B\in\sigma}-x(\card{\res{\pi}{B}}-1)!.
    \end{align}
From the interpretation of ${n\brack i}$ as counting the number of permutations of $[n]$ with $i$ cycles it is clear that ${n\brack i}=\sum_{\pi\in\pset{[n], i}}\prod_{B\in\pi}(\card{B}-1)!,$ cf.~equation (1.15) in \cite{MR2245368}. Using the isomorphism $[\pi, \rho]\cong \bigtimes_{B\in\rho} \pset{\res{\pi}{B}},$ we calculate
    \begin{align} 
    \sum_{\sigma\in [\pi,\rho]} \prod_{B\in\sigma}-x(\card{\res{\pi}{B}}-1)!
    &= \sum_{\tau'\in\bigtimes_{B\in\rho}\pset{\res{\pi}{B}}} \prod_{B\in\rho}\prod_{C\in\tau'_{B}}-x(\card{C}-1)!\notag\\
    &= \prod_{B\in\rho}\sum_{\tau\in\pset{\res{\pi}{B}}} \prod_{C\in\tau}-x(\card{C}-1)!\notag\\
    &= \prod_{B\in\rho}\left(\sum_{k=1}^{\card{\res{\pi}{B}}}(-x)^k\sum_{\tau\in\pset{\res{\pi}{B}, k}} \prod_{C\in\tau}(\card{C}-1)!\right)\notag\\
    &= \prod_{B\in\rho}\sum_{k=1}^{\card{\res{\pi}{B}}}(-x)^k {\card{\res{\pi}{B}}\brack k}\label{eq:bolthausen_term1}\\
    &= \prod_{B\in\rho}(-x)^{\overline{\card{\res{\pi}{B}}}},
  \end{align}
  where in the last step we used $x^{\overline{n}} = \sum_{k=1}^n {n\brack k}x^k,$ cf.~equation (1.16) in~\cite{MR2245368}.
\end{proof}
  
\begin{lemma}\label{lemma:bs_left_eigenmatrix}
   The matrix $L=(l_{\pi\rho})_{\pi,\rho\in\pset{[n]}}$ defined by
   \begin{align}
     l_{\pi\rho}\coloneqq\begin{cases}
     (-1)^{\card{\pi}-\card{\rho}}\frac{(\card{\rho}-1)!}{(\card{\pi}-1)!} & \text{if } \pi\leq\rho,\\
     0 & \text{otherwise,}
     \end{cases}
   \end{align}
   is the inverse matrix of $R,$ i.e.~$\sum_{\sigma\in\pset{[n]}} r_{\pi\sigma}l_{\sigma\rho}=\delta_{\pi\rho}$.
\end{lemma}
\begin{proof}
  Choosing $x=1$ in Lemma \ref{lem:technical_lemma} we have that $\sum_{\sigma\in\pset{[n]}} r_{\pi\sigma} l_{\sigma\rho} = \delta_{\pi\rho}.$
\end{proof}
\begin{proof}(of Theorem \ref{thm:bs_spectral_decomposition}) The claim follows by Lemmata \ref{lemma:bs_right_eigenvector} and \ref{lemma:bs_left_eigenmatrix}.
\end{proof}

\begin{proof}(of Corollary \ref{cor:bs_transition_probabilities})
  Since $\Pi$ is a continuous-time Markov chain with finite state space, we have for $P(t)=(p_{\pi\rho}(t))_{\pi,\rho\in\pset{[n]}}$ the identity $P(t)=\exp(tRDL)=R\exp(tD)L.$ In the last step we made use of the spectral decomposition, Theorem \ref*{thm:bs_spectral_decomposition}. In particular, this yields
  \begin{align}\label{eq:bs_transition_probabilities}
  p_{\pi\rho}(t) = \sum_{\sigma\in\pset{[n]}} r_{\pi\sigma}e^{-t(\card{\sigma}-1)}l_{\sigma\rho} =\sum_{\sigma\in [\pi,\rho]} r_{\pi\sigma}e^{-t(\card{\sigma}-1)}l_{\sigma\rho}.
  \end{align}
Letting $x=e^{-t}$ in Lemma \ref{lem:technical_lemma} proves the claim.
\end{proof}
\begin{proof}(of Corollary \ref{cor:bs_greens_matrix})
From \eqref{eq:bolthausen_term1} we have
\begin{align*}
  x^{-1}\sum_{\sigma\in [\pi,\rho]} \prod_{B\in\sigma}-x(\card{\res{\pi}{B}}-1)! &= x^{-1}\prod_{B\in\rho}\sum_{k=1}^{\card{\res{\pi}{B}}}(-x)^{k} {\card{\res{\pi}{B}}\brack k}\\
  &= \sum_{(k_B)_{B\in\rho}\in\mathbb{N}^{\card{\rho}}} (-1)^{\abs{k}}x^{\abs k-1}\prod_{B\in\rho}{\card{\res{\pi}{B}}\brack k_B},
\end{align*}
where $\abs k\coloneqq\sum_{B\in\rho}k_B.$ Letting $x=e^{-t}$ and integrating out with respect to $t$ we see for $\rho\neq\{[n]\}$
\begin{align*}
  \int_0^\infty e^{t}\sum_{\sigma\in [\pi,\rho]} \prod_{B\in\sigma}-e^{-t}(\card{\res{\pi}{B}}-1)!dt &= \sum_{(k_B)_{B\in\rho}\in\mathbb{N}^{\card{\rho}}}\frac{(-1)^{\abs k}}{\abs k-1}\prod_{B\in\rho}{\card{\res{\pi}{B}}\brack k_B}, 
\end{align*}
where for $\rho=\{[n]\}$
\begin{align*}
  \int_0^\infty e^{t}\sum_{\sigma\in [\pi,\rho]} \prod_{B\in\sigma}-e^{-t}(\card{\res{\pi}{B}}-1)!dt &= \sum_{k=1}^{\card \pi}(-1)^{k}\int_0^\infty e^{-t(k-1)}dt{\card{\pi}\brack k}= -\infty.
\end{align*}
The claim follows from \eqref{eq:bs_transition_matrix} and the definition of $g_{\pi\rho}$.
\end{proof}

\begin{proof}(of Corollary \ref{cor:bs_spectral_decomposition_block_counting_proces})
  Evidently, the rate at which a jump from $i$ to $j$ blocks occurs equals the sum of all rates at which a jump from a partition $\pi\in\pset{[n], i}$ to a partition $\rho\in\pset{[n], j}$ occurs, i.e.~$q'_{ij} = \sum_{\rho\in\pset{[n], j}} q_{\pi\rho},$ where $\pi\in\pset{{[n], i}}$ is fixed arbitrarily. The quantity $\sum_{\rho\in\pset{[n], j}} q_{\pi\rho}$ does not depend on the choice of $\pi\in\pset{[n], i},$ as the following calculation shows. By the spectral decomposition of $Q,$ Theorem \ref{thm:bs_spectral_decomposition}, we obtain
  
  \begin{align*}
    q'_{ij} &= \sum_{\rho\in\pset{[n], j}} \sum_{\sigma\in [\pi,\rho]} r_{\pi\sigma}d_{\sigma\sigma}l_{\sigma\rho}\\
    &= \sum_{\substack{\sigma\colon\pi\leq\sigma\\ \card{\sigma}\geq j}}\sum_{\substack{\rho\colon\sigma\leq\rho\\ \card{\rho}=j}} \frac{(\card{\sigma}-1)!}{(\card{\pi}-1)!}\prod_{B\in\sigma} (\card{B}-1)!(1-\card{\sigma})(-1)^{\card{\sigma}-j}\frac{(j-1)!}{(\card{\sigma}-1)!}\\
    &= \frac{(j-1)!}{(i-1)!}(-1)^j\sum_{\substack{\sigma\colon\pi\leq\sigma \\\card{\sigma}\geq j}}(\card{\sigma}-1)!\prod_{B\in\sigma} (\card{B}-1)!(1-\card{\sigma})\frac{(-1)^{\card{\sigma}}}{(\card{\sigma}-1)!}{\card{\sigma}\brace j}\\
    &= \frac{(j-1)!}{(i-1)!}(-1)^j\sum_{k=j}^i(k-1)!{i\brack k}(1-k)\frac{(-1)^k}{(k-1)!}{k\brace j}= \sum_{k=i}^j r'_{ik}d'_{kk}l'_{kj},
  \end{align*}
  hence, $Q'=R'D'L'.$
\end{proof}

\subsection{Kingman's $n$-coalescent}
For any two partitions $\pi,\rho\in\pset{[n]}$ such that $\pi\leq\rho$ one may ask in how many different ways the jump chain of Kingman's coalescent may reach $\rho$ when started in $\pi$. Since at each step only one merger of a pair of blocks occurs, there are $\card{\pi}-\card{\rho}+1$ steps to be taken, and so the set of different ways is $C(\pi, \rho)\coloneqq \{(\pi_1, \ldots, \pi_m)\colon \pi=\pi_1\lessdot\cdots\lessdot\pi_m=\rho, m=\card{\pi}-\card{\rho}+1\},$ where we defined $\lessdot$ in the paragraph preceding Theorem \ref{thm:kingman_spectral_decomposition}. We call each element in $C(\pi, \rho)$ a maximal chain in $[\pi, \rho]$ and denote by $m(\pi, \rho)\coloneqq\card{C(\pi,\rho)}$ the number of maximal chains in $[\pi, \rho]$. Before we turn to the proof of the spectral decomposition of $Q,$ Theorem \ref{thm:kingman_spectral_decomposition}, we count the number of maximal chains $m(\pi, \rho)$ in $[\pi, \rho]$ in the next Lemma.
\begin{lemma}[Number of maximal chains]\label{lemma:maximal_chains}
For $\pi, \rho\in\pset{[n]}$ with $\pi\leq\rho$ we have that
  \begin{align}\label{eq:maximal_chains}
    m(\pi, \rho) = 2^{\card{\rho}-\card{\pi}}(\card{\pi}-\card{\rho})!\prod_{B\in\rho}\card{\res{\pi}{B}}!.
  \end{align}
\end{lemma}
\begin{proof}
Notice that any maximal chain $(\pi_1, \ldots, \pi_n)$ in $[\Delta_{[n]}, \{[n]\}]$ can be constructed as follows. Let $\pi_1\coloneqq \Delta_{[n]},$ and if $\pi_i$ with $i<n$ is constructed, $\pi_{i+1}$ is obtained by merging two blocks in $\pi_i,$ which can be done in ${\card{\pi_i} \choose 2}$ ways. When $\pi_n=\{[n]\}$ is reached, the construction is finished. This construction shows that there are $m(\Delta_{[n]}, \{[n]\})=\binom{n}{2}\binom{n-1}{2}\cdots\binom{2}{2}=2^{1-n}n!(n-1)!$ maximal chains in $[\Delta_{[n]}, \{[n]\}]$. Hence \eqref{eq:maximal_chains} holds in the case $(\pi, \rho)=(\Delta_{[n]}, \{[n]\})$.

For the general case, recall the isomorphism $[\pi, \rho]\cong \bigtimes_{B\in\rho} \pset{\res{\pi}{B}}.$ As a consequence, any maximal chain in $[\pi, \rho]$ can be built by choosing a maximal chain in each factor $\pset{\res{\pi}{B}}$ and then 'intertwining' these chains, i.e.~ordering their elements (excluding the first element --- the partition into singletons --- in each chain) in any order subject to the restriction that the order of elements of the same chain is preserved. Consequently, we have
\begin{align*}
  m(\pi, \rho) &= (\card{\pi}-\card{\rho})!\prod_{B\in\rho}[(\card{\res{\pi}{B}}-1)!]^{-1}\prod_{B\in\rho}m(\Delta_{\res{\pi}{B}}, \{\res{\pi}{B}\})\\
  &= (\card{\pi}-\card{\rho})!\prod_{B\in\rho} 2^{1-\card{\res{\pi}{B}}}\card{\res{\pi}{B}}!= 2^{\card{\rho}-\card{\pi}}(\card{\pi}-\card{\rho})!\prod_{B\in\rho}\card{\res{\pi}{B}}!.
\end{align*}
\end{proof}
\begin{lemma}\label{lemma:kingman_right_eigenvector}
For any $\rho\in\pset{[n]}$ the vector $(r_{\pi\rho})_{\pi\in\pset{[n]}}$ defined by
  \begin{align}\label{eq:eigenvectors}
    r_{\pi\rho} &\coloneqq \begin{cases}
      \frac{2^{\card{\pi}-\card{\rho}}(2\card{\rho}-1)!}{(\card{\pi}-\card{\rho})!(\card{\pi}+\card{\rho}-1)!}m(\pi, \rho) & \text{if } \pi\leq\rho,\\
      0 & \text{otherwise},
    \end{cases}\\\label{eq:eigenvectors2}
    &= \begin{cases}
      \frac{(2\card{\rho}-1)!}{(\card{\pi}+\card{\rho}-1)!}\prod_{B\in\rho}\card{\res{\pi}{B}}! & \text{if } \pi\leq\rho,\\
      0 & \text{otherwise},
    \end{cases}
  \end{align}
  is a right eigenvector of $Q$ with corresponding eigenvalue $-{\card{\rho}\choose 2}.$
\end{lemma}
\begin{proof}
Fix $\pi, \rho\in\pset{[n]}.$ If $\pi<\rho,$ we have
  \begin{align*}
    &\sum_{\sigma\in\psetn}q_{\pi\sigma}r_{\sigma\rho}\\
    &= \sum_{\sigma\colon \pi\lessdot\sigma}r_{\sigma\rho}-{\card{\pi}\choose 2}r_{\pi\rho} \\
    &= \frac{2^{\card{\pi}-1-\card{\rho}}(2\card{\rho}-1)!}{(\card{\pi}-1-\card{\rho})!(\card{\pi}+\card{\rho}-2)!} \sum_{\sigma\colon\pi\lessdot\sigma}m(\sigma, \rho)-{\card{\pi}\choose 2}r_{\pi\rho}\\
    &= \left(\frac{(\card{\pi}-\card{\rho})(\card{\pi}+\card{\rho}-1)}{2}-\binom{\card{\pi}}{2}\right)r_{\pi\rho}= -{\card{\rho}\choose 2}r_{\pi\rho},
  \end{align*}
  where we used $\sum_{\sigma\colon\pi\lessdot\sigma}m(\sigma, \rho)=m(\pi, \rho).$
If $\pi=\rho,$ we have
  \begin{align*}
    \sum_{\sigma\in\pset{[n]}}q_{\pi\sigma}r_{\sigma\rho} &= q_{\pi\pi} = -{\card{\rho}\choose 2}r_{\pi\rho},
  \end{align*}
since $m(\pi, \pi)=1,$ hence $r_{\pi\pi}=1$. Finally, if $\pi\leq\rho$ does not hold, thus $r_{\pi\rho}=0$, we cannot have $\pi\leq\sigma\leq\rho$ for any $\sigma\in\pset{[n]}$ and therefore $\sum_{\sigma\in\pset{[n]}}q_{\pi\sigma}r_{\sigma\rho} = 0.$ This shows \eqref{eq:eigenvectors}. Now \eqref{eq:eigenvectors2} follows from Lemma \ref{lemma:maximal_chains} on the number of maximal chains.
\end{proof}
Evidently, the $B_n$ eigenvectors of $Q$ defined by \eqref{eq:eigenvectors} are linearly independent. We are now interested in the inverse matrix of $R= (r_{\pi\rho})_{\pi\rho\in\pset{[n]}},$ that is the matrix $L=(l_{\pi\rho})_{\pi, \rho\in\pset{[n]}}$ such that $\delta_{\pi\rho} = \sum_{\sigma\in\pset{[n]}} r_{\pi\sigma} l_{\sigma\rho}$ for all $\pi,\rho\in\pset{[n]}.$ 
\begin{lemma}\label{lemma:kingman_left_eigenvectors}
For any $\pi\in\pset{[n]}$ the vector $(l_{\pi\rho})_{\rho\in\pset{[n]}}$ given by
\begin{align}
l_{\pi\rho} &\coloneqq \begin{cases}\label{eq:kingman_lev1}
 (-1)^{\card{\pi}-\card{\rho}}\frac{2^{\card{\pi}-\card{\rho}}(\card{\pi}+\card{\rho}-2)!}{(2\card{\pi}-2)!(\card{\pi}-\card{\rho})!}m(\pi, \rho) & \text{if }\pi\leq\rho,\\
 0 & \text{otherwise,}
 \end{cases}\\
 &= \begin{cases}\label{eq:kingman_lev2}
 (-1)^{\card{\pi}-\card{\rho}}\frac{(\card{\pi}+\card{\rho}-2)!}{(2\card{\pi}-2)!}\prod_{B\in\rho}\card{\res{\pi}{B}}! & \text{if }\pi\leq\rho,\\
 0 & \text{otherwise,}
\end{cases}
\end{align}
is a left eigenvector of $Q$ with corresponding eigenvalue $-\binom{\card{\pi}}{2}$.
\end{lemma}
\begin{proof}
Use $\sum_{\sigma\colon\sigma\lessdot\rho}m(\pi, \sigma)=m(\pi, \rho)$ to obtain
\begin{align*}
    &\sum_{\sigma\in\pset{[n]}} l_{\pi\sigma}q_{\sigma\rho}\\
    &= \sum_{\sigma\colon\sigma\lessdot\rho} l_{\pi\sigma} - \binom{\card{\rho}}{2}l_{\pi\rho}\\
    &= (-1)^{\card{\pi}-\card{\rho}-1}\frac{2^{\card{\pi}-\card{\rho}-1}(\card{\pi}+\card{\rho}-1)!}{(2\card{\pi}-2)!(\card{\pi}-\card{\rho}-1)!}\sum_{\sigma\colon\sigma\lessdot\rho}m(\pi, \sigma)-\binom{\card{\rho}}{2}l_{\pi\rho}\\
    &= \left(-\frac{(\card{\pi}+\card{\rho}-1)(\card{\pi}-\card{\rho})}{2}-\binom{\card{\rho}}{2}\right)l_{\pi\rho}= -\binom{\card{\pi}}{2}l_{\pi\rho},
  \end{align*}
  thus \eqref{eq:kingman_lev1} holds. Equation \eqref{eq:kingman_lev2} follows from the Lemma on the number of maximal chains, Lemma \ref{lemma:maximal_chains}.
\end{proof}

\begin{proof}(of Theorem \ref{thm:kingman_spectral_decomposition})
The inverse matrix of $R,$ let us call it $U=(u_{\pi\rho})_{\pi,\rho\in\pset{[n]}},$ is uniquely determined, and is a matrix of left eigenvectors of $Q$, i.e.~$UQ=DU$. Moreover, for any $\pi\in\pset{[n]}$ we have by assumption $u_{\pi\pi}=\sum_{\sigma\colon\pi\leq\sigma\leq\rho}u_{\pi\sigma}r_{\sigma\rho}=\delta_{\pi\pi}=1$. This uniquely determines a matrix of left eigenvectors of $Q$, since for any $\pi, \rho\in\pset{[n]}$ with $\pi<\rho$ we have
\begin{align*}
-\binom{\card{\rho}}{2}u_{\pi\rho} = \sum_{\sigma\in [\pi, \rho]} u_{\pi\sigma}q_{\sigma\rho} = 1_{\{\pi\lessdot\rho\}} +\sum_{\sigma\colon\pi<\sigma\leq\rho} u_{\pi\sigma}q_{\sigma\rho},
\end{align*}
and $u_{\pi\rho}=0$ for $\pi\nleq\rho.$ Since $QR=RD$ by Lemma \eqref{lemma:kingman_left_eigenvectors}, and evidently $l_{\pi\pi}=1$ for any $\pi\in\pset{[n]},$ we have $U=L$ and the claim follows.
\end{proof}

\begin{remark}
  Instead of calculating the hitting probabilites $h(\pi, \rho)$ of Kingman's $n$-coalescent via the spectral decomposition, Theorem \ref{thm:kingman_spectral_decomposition}, we use the observation from section \ref{sec:results} that the jump chain of $\Pi$ can be interpreted as the directed simple random walk on $\pset{[n]}.$ This implies that the jump chain of $\Pi$ (when started from $\pi$) traces out any maximal chain in $[\pi, \{[n]\}]$ with equal probability $m(\pi, \{[n]\})^{-1}$.
  Clearly, the total number of maximal chains in $[\pi, \{[n]\}]$ that contain $\rho$ is $m(\pi, \rho)m(\rho, \{[n]\}),$ and thus
  \[ h(\pi, \rho)=\frac{m(\pi, \rho)m(\rho, \{[n]\})}{m(\pi, \{[n]\})}= \binom{\card{\pi}-1}{\card{\rho}-1}^{-1}\frac{\card{\rho}!}{\card{\pi}!}\prod_{B\in\rho}\card{\res{\pi}{B}}! = \lah{\card{\pi}}{\card{\rho}}^{-1}\prod_{B\in\rho}\card{\res{\pi}{B}}!, \]
  \end{remark}
where we used the Lemma on the number of maximal chains, Lemma \ref{lemma:maximal_chains}, in the second step. In the special case $\pi=\Delta_{[n]}$ this formula was given by Kingman in \cite{MR671034}, equation (2.3).

\begin{proof}(of Corollary \ref{cor:kingman_spectral_decomposition_block_counting_proces})
In complete analogy to the argument in Corollary \ref{cor:bs_spectral_decomposition_block_counting_proces}, we have $q'_{ij}=\sum_{\rho\in\pset{[n], j}} q_{\pi\rho}$ independent of the particular partition $\pi\in\pset{[n], i},$ as the following calculation shows. Using the spectral decomposition of $Q,$ Theorem \ref{thm:kingman_spectral_decomposition}, we obtain
\begin{align*}
  q'_{ij} &= \sum_{\rho\in\pset{[n], j}} \sum_{\sigma\in [\pi, \rho]} r_{\pi\sigma}d_{\sigma\sigma}l_{\sigma\rho}\\
  &= -\sum_{\substack{\sigma\colon \pi\leq\sigma\\ \card{\sigma}\geq j}} \frac{(2\card{\sigma}-1)!}{(i+\card{\sigma}-1)!}\left(\prod_{B\in\sigma}\card{\res{\pi}{B}}!\right)\binom{\card{\sigma}}{2}\sum_{\substack{\rho\colon\sigma\leq\rho\\\card{\rho}=j}}(-1)^{\card{\sigma}+j}\frac{(\sigma+j-2)!}{(2\card{\sigma}-2)!}\prod_{B\in\rho}\card{\res{\sigma}{B}}!\\
  &= -\sum_{k=j}^i \frac{(2k-1)!}{(i+k-1)!}\lah{i}{k} \binom{k}{2}(-1)^{k+j}\frac{(k+j-2)!}{(2k-2)!}\lah{k}{j}= \sum_{k=j}^i r'_{ik}d'_{kk}l'_{kj},
\end{align*}
where in the third step we used the identity $\lah{i}{k}=\sum_{\sigma\in\pset{[i], k}}\prod_{B\in\sigma} \card{B}!$ twice. This identity is obvious from the interpretation of $\lah{i}{k}$ as the number of partitions into $k$ ordered blocks of a set of $i$ elements, where the elements in each block are ordered.
\end{proof}

{\bf Acknowledgement.}
  This is part of the authors' PhD theses. They would like to thank their respective supervisors Martin M\"{o}hle and Alison Etheridge for their support and numerous discussions. H.~P.~thankfully acknowledges financial support by the foundation "Private Stiftung Ewald Marquardt für Wissenschaft und Technik, Kunst und Kultur."

\bibliographystyle{abbrv}\bibliography{literature}

\begin{thebibliography}{10}

\bibitem{berestycki2013}
J.~Berestycki, N.~Berestycki, and J.~Schweinsberg.
\newblock The genealogy of {B}ranching {B}rownian {M}otion with absorption.
\newblock {\em Ann. Probab.}, 41(2):527--618, 03 2013.

\bibitem{raey}
J.~Bertoin and J.-F. Le~Gall.
\newblock The {B}olthausen–{S}znitman coalescent and the genealogy of
  continuous-state branching processes.
\newblock {\em Probability Theory and Related Fields}, 117(2):249--266, 2000.

\bibitem{2013arXiv1305.6043B}
M.~{Birkner}, J.~{Blath}, and B.~{Eldon}.
\newblock {Statistical properties of the site-frequency spectrum associated
  with Lambda-coalescents}.
\newblock {\em ArXiv e-prints}, May 2013.

\bibitem{Bolthausensznitman98}
E.~Bolthausen and A.-S. Sznitman.
\newblock On \uppercase{R}uelle's probability cascades and an abstract cavity
  method.
\newblock {\em Comm. Math. Phys.}, 197(2):247--276, 1998.

\bibitem{BrunetDerrida2007}
E.~Brunet, B.~Derrida, A.~Mueller, and S.~Munier.
\newblock Effect of selection on ancestry: An exactly soluble case and its
  phenomenological generalization.
\newblock {\em Phys. Rev. E}, 76:041104, Oct 2007.

\bibitem{BrunetDerrida2006}
E.~Brunet, B.~Derrida, A.~H. Mueller, and S.~Munier.
\newblock Noisy traveling waves: Effect of selection on genealogies.
\newblock {\em EPL (Europhysics Letters)}, 76(1):1, 2006.

\bibitem{donnelly1999}
P.~Donnelly and T.~G. Kurtz.
\newblock Particle representations for measure-valued population models.
\newblock {\em The Annals of Probability}, 27(1):166--205, 01 1999.

\bibitem{EJP265}
C.~Goldschmidt and J.~Martin.
\newblock Random recursive trees and the bolthausen-sznitman coalesent.
\newblock {\em Electron. J. Probab.}, 10:no. 21, 718--745, 2005.

\bibitem{MR671034}
J.~F.~C. Kingman.
\newblock The coalescent.
\newblock {\em Stochastic Process. Appl.}, 13(3):235--248, 1982.

\bibitem{MR1880231}
M.~M{\"o}hle and S.~Sagitov.
\newblock A classification of coalescent processes for haploid exchangeable
  population models.
\newblock {\em Ann. Probab.}, 29(4):1547--1562, 2001.

\bibitem{ECP3464}
M.~Möhle and H.~Pitters.
\newblock A spectral decomposition for the block counting process of the
  {B}olthausen-{S}znitman coalescent.
\newblock {\em Electron. Commun. Probab.}, 19:no. 47, 1--11, 2014.

\bibitem{Neher08012013}
R.~A. Neher and O.~Hallatschek.
\newblock Genealogies of rapidly adapting populations.
\newblock {\em Proceedings of the National Academy of Sciences},
  110(2):437--442, 2013.

\bibitem{DBLP:books/daglib/0095301}
J.~R. Norris.
\newblock {\em Markov chains}.
\newblock Cambridge series in statistical and probabilistic mathematics.
  Cambridge University Press, 1998.

\bibitem{MR1742892}
J.~Pitman.
\newblock Coalescents with multiple collisions.
\newblock {\em Ann. Probab.}, 27(4):1870--1902, 1999.

\bibitem{MR2245368}
J.~Pitman.
\newblock {\em Combinatorial stochastic processes}.
\newblock Springer-Verlag, Berlin, 2006.

\bibitem{MR1742154}
S.~Sagitov.
\newblock The general coalescent with asynchronous mergers of ancestral lines.
\newblock {\em J. Appl. Probab.}, 36(4):1116--1125, 1999.

\bibitem{Stanley:2011:ECV:2124415}
R.~P. Stanley.
\newblock {\em Enumerative Combinatorics: Volume 1}.
\newblock Cambridge University Press, New York, NY, USA, 2nd edition, 2011.

\end{thebibliography}
\end{document}